\documentclass[a4paper,11pt]{article}
 \usepackage[plainpages=false]{hyperref}
 \usepackage{amsfonts,amsmath,amssymb,amsthm}
 \usepackage{latexsym,lscape,rawfonts,mathrsfs}

 \usepackage[dvips]{color}
 \usepackage{multicol}

 \usepackage[all]{xy}
 \usepackage{eufrak}
 \usepackage{makeidx}
 \usepackage{graphicx,psfrag}

 \usepackage{array,tabularx}

 \usepackage{setspace}

 \usepackage{appendix}

\usepackage{pdfsync}

\usepackage{txfonts}

 \newcommand{\D}[2]{\ensuremath{ \frac{\partial{#1}}{\partial{#2}}}}

 \newcommand{\R}{\ensuremath{\mathbb{R}}}
 \newcommand{\Z}{\ensuremath{\mathbb{Z}}}

 \newcommand{\ba}{\begin{align*}}
 \newcommand{\ea}{\end{align*}}

 \DeclareMathOperator{\Vol}{Vol}
 \DeclareMathOperator{\diam}{diam}

 \newcommand{\norm}[2]{{ \ensuremath{\|} #1 \ensuremath{\|}}_{#2}}

 \makeatletter
 \def\ExtendSymbol#1#2#3#4#5{\ext@arrow 0099{\arrowfill@#1#2#3}{#4}{#5}}
 
 \makeatother

 \makeatletter
 \def\ExtendSymbol#1#2#3#4#5{\ext@arrow 0099{\arrowfill@#1#2#3}{#4}{#5}}
 
 \makeatother

 \definecolor{hao}{rgb}{1,0.5,0}
 \definecolor{miao}{cmyk}{0.5,0,0.2,0.2}
 \definecolor{qiao}{gray}{0.96}

 \newtheorem*{clm}{Claim}
 
 \newtheorem{corollary}{Corollary}[section]
 \newtheorem{proposition}{Proposition}[section]
 \newtheorem{lemma}{Lemma}[section]
 
 \newtheorem{theorem}{Theorem}[section]

 \newtheorem{definition}{Definition}[section]

 \newtheorem{remark}{Remark}[section]

 \newtheorem{theoremin}{Theorem}

 \newtheorem{remarkin}{Remark}

 \title{On the conditions to extend Ricci flow(II)}
 \author{Bing Wang\footnote{Supported by NSF grant DMS-1006518.}}
 \date{}

\begin{document}
 \maketitle

 \begin{abstract}
   We develop some estimates under the Ricci flow and use these estimates to study the blowup rates of curvatures
  at singularities.  As applications, we obtain some gap theorems:
  $\displaystyle \sup_X |Ric|$ and $\displaystyle \sqrt{\sup_X |Rm|} \cdot \sqrt{\sup_X |R|}$
  must blowup at least at the rate of type-I.    Our estimates also imply some
  gap theorems for shrinking Ricci solitons.
\end{abstract}

 \section{Introduction}

 Let $X^m$ be a complete manifold of dimension $m$. $\left\{ (X^m, g(t)), 0 \leq t<T<\infty \right\}$
 is called a Ricci flow solution if $g(t)$ satisfies the equation
 \begin{align}
      \D{}{t} g(t) = -2Ric.
      \label{eqn:ricciflow}
 \end{align}
 The Ricci flow was introduced by Hamilton in his seminal paper~\cite{Ha1}, where he used the Ricci flow to study the topology of 3-manifolds with positive Ricci curvature. In the same paper, Hamilton showed the short time existence
 of equation (\ref{eqn:ricciflow}) whenever $X^m$ is a closed manifold. His proof was then simplified by DeTurck (\cite{De}). If the underlying manifold $X^m$ is a complete manifold with bounded sectional curvature, the short time existence was proved by~\cite{Shi}.

 The Ricci flow is defined as a tool to find the Einstein metric on the underlying manifold. However, generally, the Ricci flow will develop
 singularities before it converge to an Einstein metric.  A classical example is the Ricci flow starting from a
 dumbbell metric on $S^m(m \geq 3)$.  This singularity was described precisely by S.Angenent and D.Knopf(\cite{AnD}).
 Since the singularities can not be avoided, it is important to study the behavior of the Ricci flow around the singularities.

 In~\cite{Ha3}, Hamilton showed that the Ricci flow can be extended over $T$ if $|Rm|$ is uniformly bounded
 on the space-time $X \times [0, T)$. In other words,  $|Rm|$ blows up if $T$ is a singular time.
 In~\cite{Se}, Sesum proved that $|Ric|$ blows up at singular time. These theorems are fundamental.
 They were generalized in many directions.
 See~\cite{Ye},~\cite{BWa},~\cite{Kno},~\cite{MC},~\cite{EMT},~\cite{LS3},~\cite{CaZh},~\cite{CaX},
 and the references therein for more information.\\

 Before the singular time $T$ of a Ricci flow, an application of maximum principle implies that $|Rm|$ not only  blows up,
 but also blows up at a big rate(c.f.Lemma 8.7 of~\cite{CLN}):
 \begin{align}
   \lim_{t \to T} |T-t| (\sup_X |Rm|) \geq \frac{1}{8}.
   \label{eqnin:rmgap}
 \end{align}
 A natural question is: does similar behavior hold for $|Ric|$?  In this paper, we answer this question affirmatively.

 \begin{theoremin}
  Suppose $\left\{ (X, g(t)), 0 \leq t<T \right\}$ is a Ricci flow solution, $X$ is a closed manifold of dimension $m$,
  $t=T$ is a singular time.  Then
  \begin{align}
    \limsup_{t \to T} |T-t| \left(\sup_X |Ric|_{g(t)} \right) \geq \eta_1,
    \label{eqnin:et1}
  \end{align}
  where $\eta_1=\eta_1(m, \kappa)$, $\kappa$ is the non-collapsing constant of this flow.
  \label{thmin:et1}
\end{theoremin}

Note that we do not assume that the singularity is of type-I in Theorem~\ref{thmin:et1}.
If the singularity is of type-I, inequality (\ref{eqnin:et1}) was
implied by the major results in~\cite{EMT},~\cite{CaZh}, and a gap theorem of gradient shrinking solitons in~\cite{MW}.
In this case, $\eta_1$ can be chosen as $\frac{1}{100m^2}$.\\

 As indicated by~\cite{Se}, along a Ricci flow over a closed manifold,
 $|Ric|$ being uniformly bounded implies $|Rm|$ being uniformly bounded.
 One should ask whether $|Ric|$ being type-I implies $|Rm|$ being type-I?  Actually, this is a question professor
 X.X.Chen asked me around 2005.  The general answer is still open.  However, we
 can show that the blowup rate of $|Rm|$ can not be too quick if $|Ric|$ is of type-I.

\begin{theoremin}
  Suppose $\left\{ (X, g(t)), 0 \leq t<T \right\}$ is a Ricci flow solution, $X$ is a closed manifold of dimension $m$,
  $t=T$ is a singular time. If  $\displaystyle \limsup_{t \to T} |T-t|\left(\sup_X |Ric|_{g(t)}\right)=C$, then
  \begin{align*}
    \limsup_{t \to T} |T-t|^{\lambda} \left(\sup_X |Rm|_{g(t)}\right)=0
  \end{align*}
  whenever $\lambda > \frac{C}{\epsilon_1}$, where $\epsilon_1=\epsilon_1(m, \kappa)$, $\kappa$
  is the non-collapsing constant of this flow.
  \label{thmin:et2}
\end{theoremin}

\vspace{0.1in}

 It was conjectured by X.X.Chen that the Ricci flow can be extended over $T$ whenever the scalar curvature $R$
 is uniformly bounded. If the underlying manifold is K\"ahler, this conjecture was confirmed by Z.Zhang (\cite{ZhZh}).
 If the singularity is type-I, in view of the works in~\cite{EMT}, ~\cite{LS3}, and~\cite{CaZh},
 the answer is also affirmative.  However, for general Riemannian Ricci flow with dimension $m \geq 4$,
 this conjecture is still open. In this paper, we drop the type-I condition and prove the following
 gap theorem for
 $\displaystyle \limsup_{t \to T} |T-t| \left(\sqrt{\sup_X |Rm|_{g(t)}} \cdot \sqrt{\sup_X |R|_{g(t)}} \right) $.

\begin{theoremin}
  Suppose $\left\{ (X, g(t)), 0 \leq t<T \right\}$ is a Ricci flow solution, $X$ is a closed manifold of dimension $m$,
  $t=T$ is a singular time.  Then
  \begin{align}
    \limsup_{t \to T} |T-t| \left(\sqrt{\sup_X |Rm|_{g(t)}} \cdot \sqrt{\sup_X |R|_{g(t)}} \right) \geq \eta_2,
    \label{eqnin:et3}
  \end{align}
  where $\eta_2=\eta_2(m, \kappa)$, $\kappa$ is the non-collapsing constant of this flow.
  \label{thmin:et3}
\end{theoremin}

In particular, if $|R|$ is uniformly bounded, then $|Rm|$ must blowup at least at the rate
 $(T-t)^{-2}$, which imply that the singularity cannot be type-I.  Therefore, we can recover
 the extension theorems in~\cite{EMT},~\cite{LS3}, and~\cite{CaZh} by Theorem~\ref{thmin:et3}. \\

Compare Theorem~\ref{thmin:et1} and Theorem~\ref{thmin:et3}, we find that $\displaystyle \sqrt{\sup_X |Rm|_{g(t)}} \cdot \sqrt{\sup_X |R|_{g(t)}}$ behaves like
$\displaystyle \sup_X |Ric|_{g(t)}$.  Furthermore, we also have a theorem similar to Theorem~\ref{thmin:et2}.

\begin{theoremin}
  Suppose $\left\{ (X, g(t)), 0 \leq t<T \right\}$ is a Ricci flow solution, $X$ is a closed manifold of dimension $m$,
  $t=T$ is a singular time. If
   $\displaystyle \limsup_{t \to T} |T-t|\left(\sqrt{\sup_X |Rm|_{g(t)}} \cdot \sqrt{\sup_X |R|_{g(t)}} \right)=C$, then
  \begin{align*}
    \limsup_{t \to T} |T-t|^{\lambda} \left(\sup_X |Rm|_{g(t)}\right)=0
  \end{align*}
  whenever $\lambda > \frac{1}{\log_2(1+\frac{\epsilon_2}{C})}$, where $\epsilon_2=\epsilon_2(m, \kappa)$,
  $\kappa$ is the non-collapsing constant of this flow.
  \label{thmin:et4}
\end{theoremin}

\vspace{0.1in}
The proofs of Theorem~\ref{thmin:et1}, Theorem~\ref{thmin:et2}, Theorem~\ref{thmin:et3}, and Theorem~\ref{thmin:et4}
are based on two new estimates along the Ricci flow.
The first one (Theorem~\ref{thm:b23_1}) is an estimate of $|Rm|$ by integration of $|Ric|$ over a time period,
the second one (Theorem~\ref{thm:lcRm}, Corollary~\ref{cly:c3_1}, Remark~\ref{rmk:c1_2}) is an estimate of the type $|Ric| \leq \sqrt{|Rm| |R|}$.
These estimates have other applications. For example,  they yield a gap theorem for complete shrinking Ricci solitons.

\begin{theoremin}
 There exists a constant $\eta_3=\eta_3(m,\kappa)$ such that the following property holds.

  Suppose $(X^m, g)$ is complete, non-flat, $\kappa$-non-collapsed Riemannian manifold.
 If $(X^m, g)$ satisfies the shrinking Ricci soliton equation
 \begin{align}
   Ric + \mathcal{L}_V g -\frac{g}{2}=0
   \label{eqnin:ssoliton}
 \end{align}
 for some vector field $V$, $\displaystyle \sup_X |Rm|<\infty$, then
 \begin{align}
   \min \left\{\sqrt{\sup_X |Rm|} \cdot \sqrt{\sup_X |R|}, \quad \sup_X |Ric| \right\} \geq \eta_3>0.
   \label{eqnin:ssolitongap}
 \end{align}
  \label{thmin:ssolitongap}
\end{theoremin}

 This is similar to the gap theorem obtained by O. Munteanu and M.T. Wang (\cite{MW}).
 There are gap theorems of the Ricci solitons concerning different aspects of the geometry.
 For example, T. Yokota (\cite{Yok}) obtained a gap theorem concerning the
 ``reduced volume" of the gradient shrinking Ricci solitons,
 H.Z. Li (\cite{LiH}) proved a gap theorem in the K\"ahler setting.\\

 \begin{remarkin}
   A weak version of Theorem~\ref{thmin:et3} was independently obtained by X.D. Cao(\cite{CaXP}).
   He proved that $|Rm|$ blows up faster than $(T-t)^{2-\delta}$ for every $\delta>0$ whenever scalar curvature is uniformly bounded.
   \label{rmkin:xdcao}
 \end{remarkin}

 \begin{remarkin}
  The constants $\eta_1, \eta_2$ and $\eta_3$ in Theorem~\ref{thmin:et1}, Theorem~\ref{thmin:et3}, and Theorem~\ref{thmin:ssolitongap} depend only on dimension $m$.
  Also, the constants $\epsilon_1$ and $\epsilon_2$ in Theorem~\ref{thmin:et2} and Theorem~\ref{thmin:et4}
  depend only on dimension $m$. These are proved in~\cite{BWa3}.
 \label{rmkin:univC}
 \end{remarkin}

 \begin{remarkin}
  As indicated by the work of N.Le and N.Sesum (\cite{LS1}, \cite{LS2}, \cite{LS3}),
  the behaviors of the Ricci flow and the mean curvature flow are very similar.
  In our paper,  if we replace the Ricci flow by the mean curvature flow,  replace $|Rm|$ by $|A|^2$, and replace $|R|$ by $|H|^2$, then
  many theorems in this paper also hold.  For example, there is a mean curvature flow version of Theorem~\ref{thmin:et3}.
  The details of the mean curvature flow version will appear elsewhere.
   \label{rmkin:RFMCF}
 \end{remarkin}

 \vspace{0.1in}
 The organization of this paper is as follows. In section 2, we review some elementary results and fix the notations.
 In section 3, we develop the main estimates.   Then we apply these estimates to prove the extension theorems,
 gap theorems, and some other theorems in section 4.\\

 \noindent {\bf Acknowledgment}
 The author would like to thank Xiaodong Cao, Jian Song, Haozhao Li, and Yuanqi Wang for helpful discussions during the preparation of this paper.

\section{Preliminaries}
Suppose $\{(X^m, g(t)), t \in I \subset \R\}$ is a Ricci flow solution, $m \geq 4$. The curvatures evolve by the following equations(c.f.~\cite{CLN}).

\begin{align}
\begin{cases}
  \D{R}{t} &= \Delta R + 2|Ric|^2, \\
  \D{R_{ij}}{t} &= \Delta R_{ij} + 2 R_{iklj}R_{kl} - 2 R_{ik}
  R_{kj}, \\
  \D{R_{ijkl}}{t} &= \Delta R_{ijkl} +2(B_{ijkl}-B_{ijlk} +B_{ikjl} -B_{iljk})  \notag \\
   & \quad \quad       -(R_{ip}R_{pjkl}
   +R_{jp}R_{ipkl}+R_{kp}R_{ijpl}+R_{lp}R_{ijkp}),
\end{cases}
\end{align}
where $B_{ijkl} \triangleq -R_{ipqj} R_{kpql}$.  It follows that
\begin{align}
\begin{cases}
  & \left( \D{}{t}- \Delta \right) R = 2|Ric|^2, \\
  & \left( \D{}{t}- \Delta \right) |Ric|^2 \leq 4|Rm||Ric|^2, \\
  & \left( \D{}{t}- \Delta \right) |Rm|^2 \leq 16|Rm|^3.
\end{cases}
\label{eqn:curvinequa}
\end{align}
We will use inequalities (\ref{eqn:curvinequa}) for the purpose of Moser iteration.\\

For simplicity of notations, we give some definitions.
\begin{definition}
  A Riemannian manifold $(X^m, g)$ is called $\kappa$-non-collapsed if for every geodesic ball $B(x, r) \subset X$ with the property
  $\displaystyle \sup_{B(x, r)} |Rm| \leq r^{-2}$, we have  $\Vol(B(x,r)) \geq \kappa r^m$.
  \label{dfn:knoncollapse}
\end{definition}

\begin{definition}
  Along the Ricci flow $\left\{ (X, g(t)), t \in I \right\}$ ($I$ is a connected interval in $\R$), define
  \begin{align*}
    O_g(t)=\sup_X |R|_{g(t)}, \quad P_g(t)=\sup_X |Ric|_{g(t)}, \quad Q_g(t)=\sup_X |Rm|_{g(t)}.
  \end{align*}
  We may omit the subindex ``$g$" if the flow is obvious in the content.
  \label{dfn:k30_1}
\end{definition}

\begin{definition}
 Define $\mathcal{L}(m,\kappa, I)$ be the moduli space of the Ricci flows $\left\{ (X, g(t)) | t \in I  \right\}$
satisfying the following properties.
\begin{itemize}
  \item $X$ is a complete Riemannian manifold of dimension $m$.
  \item $Q_g(t) <\infty$ for every $t \in I$.
  \item $(X, g(t))$ is $\kappa$-non-collapsed for every $t \in I$.
\end{itemize}
  \label{dfn:b23_2}
\end{definition}

\begin{definition}
 Define $\mathcal{M}(m,\kappa, I)$ be the moduli space of the Ricci flows $\left\{ (X, g(t)) | t \in I  \right\}$
satisfying the following properties.
\begin{itemize}
  \item $X$ is a complete Riemannian manifold of dimension $m$.
  \item $Q_g(t) \leq 2$ for every $t \in I$.
  \item $\displaystyle \lim_{t \to b} Q_g(t) \geq \frac{1}{2}$ where $b =\sup\left\{ t | t\in I \right\}$.
  \item $(X, g(t))$ is $\kappa$-non-collapsed for every $t \in I$.
\end{itemize}

Define $\mathcal{M}(m, \kappa)=\mathcal{M}(m, \kappa, [-1, 0])$.
  \label{dfn:b23_1}
\end{definition}

\begin{definition}
  Suppose $\left\{(X, g(t)), 0 \leq t<T<\infty \right\}$ is a Ricci flow solution with singular time $T$.
  A quantity $f$ (which may be $|Rm|, |Ric|, |R|$)
  is called of type-I if
  \begin{align*}
    \limsup_{t \to T} |T-t| \sup_X |f(\cdot, t)| <\infty.
  \end{align*}
  $f$ is called at least of type-I if
  \begin{align*}
    \limsup_{t \to T} |T-t| \sup_X |f(\cdot, t)| >0.
  \end{align*}
  The singular time is called of type-I if $|Rm|$ is of type-I.
  \label{dfn:typeI}
\end{definition}

\begin{definition}[GH-distance]
  Suppose $Z$ is a metric space, $A_1, A_2$ are two subsets of $Z$, then the Hausdorff distance, $d_H$ is
  \begin{align*}
    d_H(A_1, A_2)=\inf\left\{ r | A_2 \subset B(A_1, r), \;\textrm{and} \; A_1 \subset B(A_2, r) \right\}.
  \end{align*}

  Suppose $X$ and $Y$ are two metric spaces, the Gromov-Hausdorff distance is defined as
  \begin{align*}
    &\qquad d_{GH}(X, Y)\\
    &= \inf\left\{ d_H(i(X), j(Y)) | i:\; X \to Z, \quad j: Y \to Z \; \textrm{are isometric embeddings}, \; Z \; \textrm{is a metric space} \right\}.
  \end{align*}

  Suppose $X$ is a metric space with base point $x$, $Y$ is a metric space with base point $y$, the pointed-Gromov-Hausdorff distance
  is defined as
  \begin{align*}
   &\qquad d_{GH} \left(  (X, x), (Y, y) \right) \\
   &=\inf \left\{ r|
    \textrm{There exist a metric space $Z$ and isometric embeddings} \; i: B(x, \frac{1}{r}) \to Z, \; j: B(y, \frac{1}{r}) \to Z, \right. \\
   &\qquad \qquad \left.  \textrm{such that} \; d_{H} \left( B(i(x), \frac{1}{r}),  B(j(y), \frac{1}{r})\right)<r,
    \quad d(i(x), j(y))<r
     \right\}.
  \end{align*}
  \label{dfn:GHdist}
\end{definition}

\begin{definition}[$\epsilon$-approximation]
  Suppose $(X, x)$ is a metric space with base point $x$,
  $(Y, y)$ is a metric space with base point $y$.
  A map $\varphi: (X, x) \to (Y, y)$ is called an  $\epsilon$-approximation if
  \begin{itemize}
    \item   $d(\varphi(x), y)<\epsilon$.
    \item   $B(y, \frac{1}{\epsilon}) \subset \left\{ y \in Y | d(y, \varphi(B(x, \frac{1}{\epsilon})))
      <\epsilon \right\}$.
    \item   $|d(x_1, x_2) -d(\varphi(x_1), \varphi(x_2))|<\epsilon$ for every $x_1, x_2 \in B(x, \frac{1}{\epsilon})$.
  \end{itemize}
 \label{dfn:approximation}
\end{definition}

 It is not hard to see(c.f.~\cite{GLP}) the following property.
 \begin{proposition}

 \begin{itemize}
   \item If $d_{GH}((X, x), (Y, y))<\epsilon$, then there exists a $(10\epsilon)$-approximation $\varphi: X \to Y$.
   \item If there is an $\epsilon$-approximation $\varphi: (X, x) \to (Y, y)$, then $d_{GH}( (X, x), (Y, y))<10\epsilon$.
 \end{itemize}
 \label{prn:GHapprox}
 \end{proposition}

\section{Some Curvature Estimates along the Ricci flow}

\subsection{Estimate Riemannian curvature by integration of Ricci curvature on a time period.}

\begin{lemma}
  For every $\delta>0$, there exists an $\epsilon=\epsilon(m, \kappa, \delta)$ such that the following property holds.

  If $\mathbf{G}=\left\{ (X, x_0, g(t)), -1 \leq t  \leq 0 \right\}$,
  $\mathbf{H}=\left\{ (Y, y_0, h(t)), -1 \leq t \leq 0 \right\}$
  are two Ricci flows in the moduli space $\mathcal{M}(m, \kappa)$ satisfying
  \begin{itemize}
    \item  $d_{GH}\left\{ (\Omega_X, x_0, g(0)),  (\Omega_Y, y_0, h(0)) \right\}<\epsilon$,  where $\Omega_X=\overline{B_{g(0)}(x_0, 1)}$, $\Omega_Y=\overline{B_{h(0)}(y_0, 1)}$,
    \item  $\displaystyle \max\left\{ |Rm|_{g(0)}(x_0),  |Rm|_{h(0)}(y_0) \right\} \geq \frac{1}{2}$,
  \end{itemize}
  then we have
  \begin{align*}
    \left| \log \frac{|Rm|_{h(0)}(y_0)}{|Rm|_{g(0)}(x_0)} \right| <  \delta.
  \end{align*}
  \label{lma:edric1}
\end{lemma}

\begin{proof}
  If this lemma was wrong, then there exists a constant $\delta_0>0$ such that no matter how small $\epsilon$ is, one can have two
  Ricci flows in $\mathcal{M}(m, \kappa)$ violating the given property with couple $(\delta_0, \epsilon)$.
  Therefore,  there exist two sequences of Ricci flows
  \begin{align*}
    &\mathbf{G}_i =\left\{ (X_i, x_i, g_i(t)), -1 \leq t \leq 0 \right\} \in \mathcal{M}(m, \kappa), \\
    &\mathbf{H}_i =\left\{ (Y_i, y_i, h_i(t)), -1 \leq t \leq 0 \right\} \in \mathcal{M}(m, \kappa),
  \end{align*}
  such that
  \begin{align}
    \begin{cases}
    &\max\left\{ |Rm|_{g_i(0)}(x_i),  |Rm|_{h_i(0)}(y_i) \right\} \geq \frac{1}{2}, \\
    &\left| \log \frac{|Rm|_{h_i(0)}(y_i)}{|Rm|_{g_i(0)}(x_i)} \right| \geq  \delta_0,\\
    &d_{GH}\left\{ (\Omega_{X_i}, x_i, g_i(0)),  (\Omega_{Y_i}, y_i, h_i(0)) \right\}<\epsilon_i \to 0.
    \end{cases}
    \label{eqn:dbig}
  \end{align}
  By the curvature bound and $\kappa$-non-collapsing condition, we can apply Hamilton's compactness theorem
  to obtain smooth convergence(c.f.~\cite{Ha2}):
  \begin{align*}
    &\left\{ (X_i, x_i, g_i(t)), -1 < t \leq 0 \right\} \stackrel{C^{\infty}}{\longrightarrow}
    \left\{ (X_{\infty}, x_{\infty}, g_{\infty}(t)), -1 < t \leq 0 \right\},\\
    &\left\{ (Y_i, y_i, h_i(t)), -1 < t \leq 0 \right\} \stackrel{C^{\infty}}{\longrightarrow}
    \left\{ (Y_{\infty}, y_{\infty}, h_{\infty}(t)), -1 < t \leq 0 \right\}.
  \end{align*}
  Since $d_{GH}\left\{ (\Omega_{X_i}, x_i, g_i(0)),  (\Omega_{Y_i}, y_i, h_i(0)) \right\}<\epsilon_i \to 0$, we have
  \begin{align*}
    d_{GH}\left\{ (\Omega_{X_{\infty}}, x_{\infty}, g_{\infty}(0)),
       (\Omega_{Y_{\infty}}, y_{\infty}, h_{\infty}(0)) \right\}=0.
  \end{align*}
  Denote $\Omega_{\infty}=\Omega_{X_{\infty}}=\Omega_{Y_{\infty}}$. Note that $\Omega_{\infty}$ is a smooth unit geodesic ball with
  center $p_{\infty}=x_{\infty}=y_{\infty}$.  By the smooth convergence, we see that
  \begin{align*}
    \begin{cases}
     &\lim_{i \to \infty} |Rm|_{h_i(0)}(y_i)=|Rm|(y_{\infty})=|Rm|(p_{\infty})=|Rm|(x_{\infty})=\lim_{i \to \infty} |Rm|_{g_i(0)}(x_i), \\
     &\frac{1}{2} \leq |Rm|(p_{\infty}) \leq 2.
    \end{cases}
  \end{align*}
  It follows that
  \begin{align*}
    \lim_{i \to \infty} \frac{|Rm|_{h_i(0)}(y_i)}{|Rm|_{g_i(0)}(x_i)}=1, \quad
    \Longrightarrow \quad
    \lim_{i \to \infty} \left| \log \frac{|Rm|_{h_i(0)}(y_i)}{|Rm|_{g_i(0)}(x_i)} \right| =0.
  \end{align*}
  Therefore, $\left| \log \frac{|Rm|_{h_i(0)}(y_i)}{|Rm|_{g_i(0)}(x_i)} \right|<\delta_0$ for large $i$.  This contradicts to
  the condition (\ref{eqn:dbig})!
\end{proof}

\begin{figure}[h]
 \begin{center}
  \psfrag{t-axis}[c][c]{$t$}
  \psfrag{t-1}[l][l]{$t=-1$}
  \psfrag{t0}[c][c]{$t=0$}
  \psfrag{t1}[c][c]{$t=t_1$}
  \psfrag{tk}[c][c]{$t=K$}
  \psfrag{X}[l][l]{$X$}
  \psfrag{x0}[c][c]{$x_0$}
  \psfrag{y0}[c][c]{$y_0$}
  \psfrag{q+}[c][c]{$\frac12 \leq Q(t) \leq 2$}
  \psfrag{q-}[c][c]{$0 \leq Q(t) \leq 2$}
  \psfrag{dgh}[c][c]{$d_{GH}(B_{g(0)}(z_0,1), B_{g(t_1)}(z_0,1))>\epsilon$}
  \includegraphics[width=0.8\columnwidth]{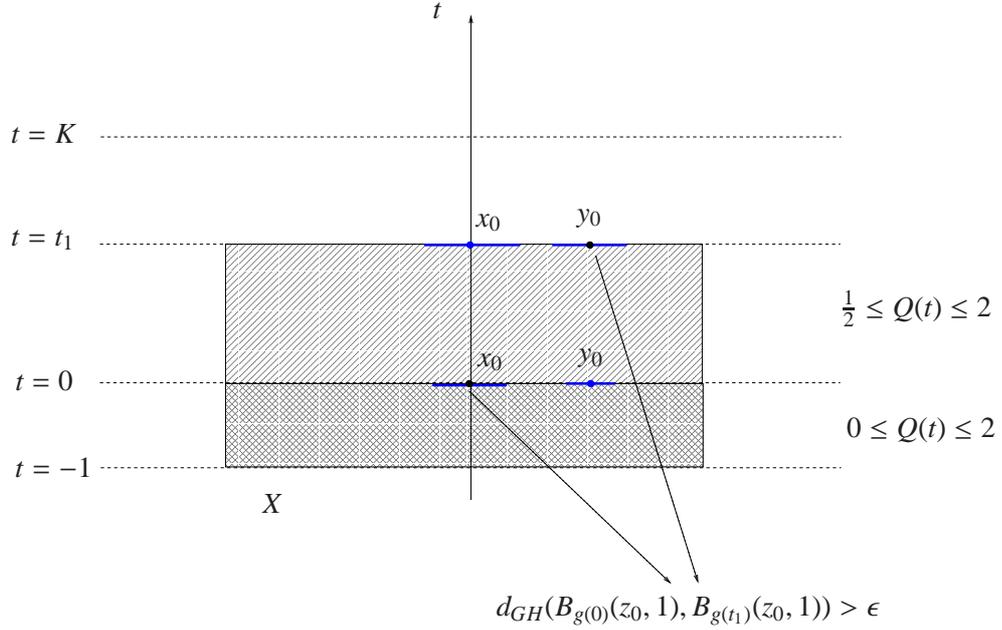}
  \caption{Gap of $\int P(t)dt$}
  \label{fig:Ricintgap}
 \end{center}
\end{figure}

\begin{lemma}
  There exists a constant $\epsilon_0=\epsilon_0(m, \kappa)$ such that the following property holds.

  Suppose $K \geq 0$, $\left\{ (X, g(t)), -1 \leq t \leq K \right\} \in \mathcal{L}(m, \kappa, [-1, K])$,
  $Q_g(0)=1$, and $Q_g(t)\leq 2$ for every $t \in [-1, 0]$.
  If $t_1 > 0$ is the first time such that $|\log Q_g(t)|=\log 2$, then
  \begin{align}
    \int_0^{t_1} P_g(t)dt > \epsilon_0.
  \label{eqn:b23_11}
  \end{align}
  \label{lma:b23_1}
\end{lemma}

\begin{proof}
  For simplicity of notation, let $g$ be the default flow. So  we denote $Q_g(t)$ by $Q(t)$,
  $P_g(t)$ by $P(t)$, etc.

  Choose $x_0, y_0\in X$ such that
  \begin{align*}
   & |Rm|_{g(0)}(x_0) \geq \frac{9}{10}Q(0)=\frac{9}{10}.\\
   & |Rm|_{g(t_1)}(y_0) \geq \frac{9}{10}Q(t_1), \Rightarrow \frac{9}{5} \leq |Rm|_{g(t_1)}(y_0) \leq 2,
   \; \textrm{or} \; \frac{9}{20} \leq|Rm|_{g(t_1)}(y_0) \leq \frac{1}{2}.
 \end{align*}
  Define $h(t)=g(t-t_1)$. Clearly, we have
  \begin{align*}
    & \left\{ (X, x_0, g(t)), -1 \leq t  \leq 0 \right\} \in \mathcal{M}(m,\kappa), \\
    & \left\{ (X, y_0, h(t)), -1 \leq t \leq 0  \right\} \in \mathcal{M}(m,\kappa).
  \end{align*}
  If $|Rm|_{g(t_1)}(y_0) \in [\frac{9}{5}, 2]$, then
  \begin{align*}
    \frac{|Rm|_{h(0)}(y_0)}{|Rm|_{g(0)}(y_0)}=\frac{|Rm|_{g(t_1)}(y_0)}{|Rm|_{g(0)}(y_0)}\geq \frac{|Rm|_{g(t_1)}(y_0)}{Q(0)} \geq \frac{9}{5},
    \quad |Rm|_{h(0)}(y_0) =|Rm|_{g(t_1)}(y_0) \geq \frac{1}{2}.
  \end{align*}
  If $|Rm|_{g(t_1)}(y_0) \in [\frac{9}{20}, \frac{1}{2}]$, then
  \begin{align*}
    \frac{|Rm|_{g(0)}(x_0)}{|Rm|_{h(0)}(x_0)} =\frac{|Rm|_{g(0)}(x_0)}{|Rm|_{g(t_1)}(x_0)} \geq \frac{|Rm|_{g(0)}(x_0)}{Q(t_1)}
    =2|Rm|_{g(0)}(x_0) \geq \frac{9}{5}, \quad |Rm|_{g(0)}(x_0) \geq \frac{9}{10} >\frac{1}{2}.
  \end{align*}
  So there exists a point $z_0 \in X$, which may be $x_0$ or $y_0$ such that the following properties hold.
  \begin{align*}
    \max\left\{ |Rm|_{g(0)}(z_0), |Rm|_{h(0)}(z_0) \right\} \geq \frac{1}{2}, \quad
    \left|\log \frac{|Rm|_{g(0)}(z_0)}{|Rm|_{h(0)}(z_0)} \right| \geq \log \frac{9}{5}.
  \end{align*}
  By Lemma~\ref{lma:edric1}, there exists a constant $\epsilon=\epsilon(m, \kappa)$ such that
   \begin{align}
    d_{GH}\left\{ (B_{g(0)}(z_0, 1), z_0, g(0)),  (B_{h(0)}(z_0, 1), z_0, h(0)) \right\}>\epsilon.
  \label{eqn:b23_9}
  \end{align}
  Choose every two points $w,z \in X$, we have
  \begin{align}
    \left|\log  \frac{d_{h(0)}(w,z)}{d_{g(0)}(w,z)} \right|=
    \left|\log  \frac{d_{g(t_1)}(w,z)}{d_{g(0)}(w,z)} \right| \leq  \int_0^{t_1} P(t)dt.
    \label{eqn:b23_1}
  \end{align}

  In order to finish our proof, it suffices to prove the following Claim.

  \begin{clm}
  Inequalities (\ref{eqn:b23_9}) and (\ref{eqn:b23_1}) imply
  that there exists a constant $\epsilon_0=\epsilon_0(m,\kappa)$ such that
  \begin{align}
    \int_0^{t_1} P(t)dt >\epsilon_0.
  \label{eqn:clmRicgap}
  \end{align}
  \end{clm}

  Let $\xi=\int_0^{t_1} P(t)dt$. Equation (\ref{eqn:b23_1}) becomes
  \begin{align}
    \left|\log  \frac{d_{g(t_1)}(w, z)}{d_{g(0)}(w, z)} \right| \leq \xi
  \label{eqn:dquotientxi}
  \end{align}
  for every two points $w, z \in X$. It follows that
  \begin{align*}
   B_{g(t_1)}(z_0, e^{-2\xi}) \subset B_{g(0)}(z_0, e^{-\xi}) \subset B_{g(t_1)}(z_0, 1).
  \end{align*}
  Let $\Omega_a=B_{g(0)}(z_0, 1), \; \Omega_b=B_{g(t_1)}(z_0, 1), \; \Omega'=B_{g(0)}(z_0, e^{-\xi})$.
  Therefore, we have
  \begin{align}
    &\qquad d_{GH}\left\{ (\Omega_a, z_0,  g(0)),  (\Omega_b, z_0, g(t_1)) \right\}  \nonumber\\
    &\leq d_{GH}\left\{ (\Omega_a, z_0, g(0)), (\Omega', z_0, g(0)) \right\}
    +d_{GH} \left\{ (\Omega', z_0, g(0)),  (\Omega', z_0, g(t_1)) \right\} \nonumber \\
    &\quad +d_{GH} \left\{ (\Omega', z_0, g(t_1)),  (\Omega_b, z_0, g(t_1)) \right\}.
    \label{eqn:c6_0}
  \end{align}
  Since $(\Omega_a, g(0)), \; (\Omega', g(0))$ are two sub-metric-spaces of $(X, g(0))$, $\Omega' \subset \Omega_a$, and  $\Omega_a$
  is in the $(1-e^{-\xi})$-neighborhood
  of $\Omega'$, therefore by the definition of GH-distance, we have
  \begin{align}
    d_{GH} \left\{ (\Omega_a, z_0, g(0)),  (\Omega', z_0, g(0))  \right\} < 1-e^{-\xi}.
    \label{eqn:c6_1}
  \end{align}
  Note that $B_{g(t_1)}(z_0, e^{-2\xi}) \subset \Omega' \subset \Omega_b$, we obtain
  \begin{align}
    d_{GH}\{ (\Omega', z_0, g(t_1)), (\Omega_b, z_0, g(t_1))\} < 1-e^{-2\xi}.
    \label{eqn:c6_2}
  \end{align}
  Consider the identity map:
  \begin{align*}
    Id: (\Omega', g(0)) &\mapsto (\Omega', g(t_1)), \\
             x &\mapsto \; Id(x)=x.
  \end{align*}
  By inequality (\ref{eqn:dquotientxi}) and the fact that
  \begin{align*}
    \max \left\{ \diam_{g(0)}(\Omega'), \; \diam_{g(t_1)}(\Omega') \right\}<2,
  \end{align*}
  we see that $Id$ is a $2(1-e^{-\xi})$-approximation of $(\Omega', z_0, g(0))$
  and $(\Omega', z_0, g(t_1))$. It follows that
  \begin{align}
    d_{GH} \left\{  (\Omega', z_0, g(0)), (\Omega', z_0, g(t_1)) \right\} <20(1-e^{-\xi}).
    \label{eqn:c6_3}
  \end{align}
  Combine inequalities (\ref{eqn:c6_0}), (\ref{eqn:c6_1}), (\ref{eqn:c6_2}),
  and (\ref{eqn:c6_3}), we obtain
  \begin{align}
    d_{GH}\left\{ (\Omega_a, z_0, g(0)),  (\Omega_b, z_0, g(t_1)) \right\}<21(1-e^{-\xi}) + (1-e^{-2\xi}).
    \label{eqn:c6_4}
  \end{align}
  From inequality (\ref{eqn:b23_9}), (\ref{eqn:b23_1}), and (\ref{eqn:c6_4}), we obtain
  \begin{align}
    \epsilon< 21(1-e^{-\xi}) + (1-e^{-2\xi}).
   \label{eqn:c6_5}
  \end{align}
  This forces $\xi> \epsilon_0=\epsilon_0(\epsilon(m, \kappa))=\epsilon_0(m, \kappa)$.
  So we finish the proof of the Claim.
\end{proof}

\begin{theorem}
  Suppose $K \geq 0$, $\left\{ (X, g(t)), -1 \leq t \leq K \right\} \in \mathcal{L}(m, \kappa, [-1, K])$,
  $Q(0)=1$, and $Q(t)\leq 2$ for every $t \in [-1, 0]$. Then we have
  \begin{align}
    Q(K)< 2^{\frac{\int_0^K P(t)dt}{\epsilon_0}+1},
  \label{eqn:b23_10}
  \end{align}
  where $\epsilon_0=\epsilon_0(m, \kappa)$ is the constant in Lemma~\ref{lma:b23_1}.
  \label{thm:b23_1}
\end{theorem}

\begin{proof}
  Define  $s_i=\inf\left\{ t|t\geq 0, Q(t)=2^i \right\}$ for every nonnegative integer $i$.
  Clearly, $s_0=0$.  According to the choice of $s_i$, we have
  \begin{align*}
    s_i-Q^{-1}(s_i) \geq -1, \quad \sup_{X \times [ s_i-Q^{-1}(s_i), s_{i+1}]} |Rm|= Q(s_{i+1})
    =2Q(s_i)=2^{i+1}Q(0)=2^{i+1}.
  \end{align*}
  Let $g_i(t)=Q(s_i)g(Q^{-1}(s_i)t +s_i)$, then the flow
  $\{(X, g_i(t)), -1 \leq t \leq Q(s_i)(K-s_i)\}$ satisfies all the conditions in Lemma~\ref{lma:b23_1}.
  Note that $Q_{g_i}(Q(s_i)(s_{i+1}-s_i))=2$, it follows from Lemma~\ref{lma:b23_1} that
  \begin{align*}
    \int_{s_i}^{s_{i+1}} P(t)dt = \int_0^{Q(s_i)(s_{i+1}-s_i)} P_{g_i}(t)dt > \epsilon_0.
  \end{align*}
  Let $N$ be the largest $i$ such that $s_i \leq K$, we have
  \begin{align*}
    N\epsilon_0 < \int_0^{s_N} P(t)dt \leq \int_0^K P(t)dt, \quad \Rightarrow N<\frac{\int_0^K P(t)dt}{\epsilon_0}.
  \end{align*}
  For every $t \in [0, K]$, we obtain that
  \begin{align*}
   Q(t) \leq 2^{N+1} < 2^{\frac{\int_0^K P(t)dt}{\epsilon_0}+1}.
  \end{align*}
\end{proof}

\begin{corollary}
   Suppose $\left\{ (X, g(t)), -1 \leq t <0 \right\} \in \mathcal{L}(m, \kappa, [-1, 0))$,  $t=0$ is the singular time. Then
\begin{align}
  \int_{-1}^{0} P(t) dt =\infty.
\end{align}
  \label{cly:a7_1}
\end{corollary}

\begin{proof}
  This follows from Theorem~\ref{thm:b23_1} and the fact $\displaystyle \lim_{t \to 0}Q(t)=\infty$.
\end{proof}

  \begin{corollary}
    Suppose $\left\{ (X, g(t)), -1 \leq t <0 \right\} \in \mathcal{L}(m, \kappa, [-1, 0))$,  $t=0$ is the singular time.
   If $\displaystyle \limsup_{t \to 0} P(t)|t|=C$, then we have
   \begin{align}
     Q(t) = o(|t|^{-\lambda}),
     \label{eqn:qspeed}
   \end{align}
   whenever $\lambda>\frac{C\log 2}{\epsilon_0}$
   with $\epsilon_0=\epsilon_0(m, \kappa)$ being  the constant in Theorem~\ref{thm:b23_1}.
   \label{cly:b24_1}
 \end{corollary}

\begin{proof}
  Fix $\lambda>\frac{C\log 2}{\epsilon_0}$. Choose $\delta>0$ such that
  \begin{align}
    \lambda>\frac{(C+\delta)\log 2}{\epsilon_0}.
    \label{eqn:b24_0}
  \end{align}
  Since $\displaystyle \limsup_{t \to 0} P(t)|t|=C$, $\displaystyle \lim_{t \to 0}Q(t)=\infty$,
  we can choose $t_0=t_0(g, \delta)$ such that
  \begin{align}
    \begin{cases}
     &P(t)|t|<C+\delta, \quad \forall \; t \in [t_0, 0], \\
     &Q(t) \leq Q(t_0), \quad \forall \; t \in [-1, t_0],\\
     &Q(t_0)|1+t_0| \geq 1.
    \end{cases}
    \label{eqn:b24_1}
  \end{align}
  Define $s_i=\inf \left\{ t |t \geq t_0, Q(t)=2^i Q(t_0) \right\}$.  Clearly, $s_0=t_0$.

  Let $h_i(t)=Q(s_i)g(Q^{-1}(s_i)t + s_i)$, we can truncate a flow
  $\left\{ (X, h_i(t)), -1 \leq t <Q(s_i)|s_i| \right\}$ which satisfies the following properties.
  \begin{align}
    \begin{cases}
    &\left\{ (X, h_i(t)), -1 \leq t <Q(s_i)|s_i| \right\} \in \mathcal{L}(m, \kappa, [-1, Q(s_i)|s_i|)).\\
    &Q_{h_i}(0)=1,\quad Q_{h_i}(t) \leq 1, \; \forall \; t \in [-1, 0].\\
    &Q_{h_i}(Q(s_i)|s_{i+1}-s_i|)=2.
    \end{cases}
    \label{eqn:b24_2}
  \end{align}
  Therefore, Lemma~\ref{lma:b23_1} implies that
  \begin{align}
    \int_{s_i}^{s_{i+1}} P(t)dt = \int_0^{Q(s_i)|s_{i+1}-s_i|} P_{h_i}(t) dt>\epsilon_0.
    \label{eqn:b24_3}
  \end{align}
  On the other hand, we have
  \begin{align}
    P(t)< \frac{C+\delta}{|t|}, \; \Rightarrow \; \int_{s_i}^{s_{i+1}} P(t)dt < (C+\delta) \log \frac{|s_i|}{|s_{i+1}|}.
    \label{eqn:b24_4}
  \end{align}
  Combining inequality (\ref{eqn:b24_3}) and (\ref{eqn:b24_4}), we obtain
  \begin{align}
    \log \frac{|s_i|}{|s_{i+1}|} > \frac{\epsilon_0}{C+\delta}, \; \Rightarrow \;  \frac{|s_{i+1}|}{|s_i|} < e^{-\frac{\epsilon_0}{C+\delta}},\;
    \Rightarrow \; |s_i| < |s_0|e^{-\frac{i\epsilon_0}{C+\delta}}=|t_0|e^{-\frac{i\epsilon_0}{C+\delta}}.
    \label{eqn:b24_5}
  \end{align}
  Therefore, we have
  \begin{align}
    \lim_{i \to \infty} Q(s_i)|s_i|^{\lambda} \leq  \lim_{i \to \infty} Q(t_0)|t_0|^{\lambda}\left( 2 e^{-\frac{\lambda \epsilon_0}{C+\delta}} \right)^i=0.
    \label{eqn:b24_6}
  \end{align}
  The equality of (\ref{eqn:b24_6}) holds since equation (\ref{eqn:b24_0}) implies $2 e^{-\frac{\lambda \epsilon_0}{C+\delta}}<1$.

  Note that $s_i<s_{i+1}<0$. For every $t \in [s_i, s_{i+1}]$, we have $|s_{i+1}| \leq |t| \leq |s_i|$.  It follows that
  \begin{align*}
    Q(t)|t|^{\lambda} \leq Q(s_{i+1})|s_i|^{\lambda}=2Q(s_i)|s_i|^{\lambda} \to 0, \quad \textrm{as} \; i \to \infty.
  \end{align*}
  This yields that
  \begin{align*}
    \limsup_{t \to 0} Q(t)|t|^{\lambda}=0,  \; \Leftrightarrow \; Q(t)=o(|t|^{-\lambda}).
  \end{align*}

\end{proof}

\begin{corollary}
  Suppose $\left\{ (X, g(t)), -1 \leq t <0 \right\} \in \mathcal{L}(m, \kappa, [-1, 0))$,  $t=0$ is the singular time.
  Then
\begin{align}
  \limsup_{t \to 0} P(t)|t| \geq \frac{\epsilon_0}{\log 2},
  \label{eqn:b24_7}
\end{align}
  where $\epsilon_0=\epsilon_0(m, \kappa)$ is the constant in Theorem~\ref{thm:b23_1}.
  \label{cly:a6_2}
\end{corollary}

\begin{proof}
  Suppose $\displaystyle \limsup_{t \to 0} P(t)|t|=C \geq 0$.
  By Corollary~\ref{cly:b24_1} and the fact $\displaystyle Q(t)|t| \geq \frac{1}{8}$,
  we have
  \begin{align*}
    0= \limsup_{t \to 0} Q(t)|t|^{\frac{(C+\delta)\log 2}{\epsilon_0}} \geq \limsup_{t \to 0} \frac{1}{8} |t|^{-1+\frac{(C+\delta)\log 2}{\epsilon_0}},
  \end{align*}
  for every $\delta>0$. It follows that
  \begin{align*}
    -1+\frac{C\log 2}{\epsilon_0} \geq 0, \quad \Rightarrow \quad C \geq \frac{\epsilon_0}{\log 2}.
  \end{align*}
\end{proof}

 \subsection{Estimate Ricci curvature  by Riemannian curvature and scalar curvature.}

 \begin{figure}[h]
 \begin{center}
  \psfrag{t-axis}[c][c]{$t$}
  \psfrag{t-1}[l][l]{$t=-1$}
  \psfrag{t0}[c][c]{$t=0$}
  \psfrag{t1}[c][c]{$t=t_1$}
  \psfrag{X}[l][l]{$X$}
  \psfrag{x0}[c][c]{$x_0$}
  \psfrag{B1}[c][c]{$B_{g(0)}(x_0,\frac12) \times [-\frac12,0]$}
  \psfrag{B2}[c][c]{$B_{g(0)}(x_0,10) \times [-1,0]$}
  \psfrag{B3}[c][c]{$\{(x,t)|d_{g(t)}(x,x_0)<100, -1 \leq t \leq 0\}$}
  \psfrag{Rminj}[c][c]{$|Rm|\leq \frac{1}{m^2}, \;inj(x_0) \geq 3.$}
  \includegraphics[width=0.8\columnwidth]{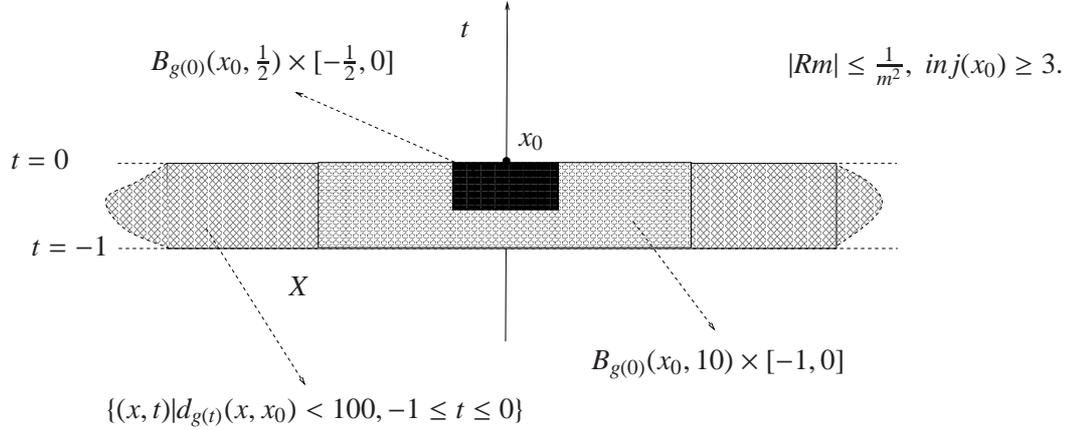}
  \caption{Estimate $|Ric|$ by $|Rm|$ and $|R|$}
  \label{fig:RicRmR}
 \end{center}
\end{figure}

 \begin{theorem}
  Suppose $\left\{ (X, g(t)), -1 \leq t \leq 0 \right\}$ is a Ricci flow solution satisfying  the following properties.
  \begin{itemize}
    \item $X$ is a complete manifold of dimension $m$.
    \item $|Rm|_{g(t)}(x) \leq \frac{1}{m^2}$ whenever $x \in B_{g(t)}(x_0, 100), \; t \in [-1, 0]$.
    \item $(B_{g(0)}(x_0,100), g(t))$ has a uniform Sobolev constant $\sigma$ for every $-1 \leq t \leq 0$.
    \item $inj_{g(t)}(x_0) \geq 3$ uniformly for every $-1 \leq t \leq 0$.
  \end{itemize}
  Then there exists a large constant $C=C(m, \sigma)$ such that
  \begin{align}
    \sup_{B_{g(0)}(x_0, \frac{1}{2}) \times [-\frac{1}{2}, 0]} |Ric|
    \leq C \norm{R}{L^{\infty}(B_{g(0)}(x_0, 10) \times [-1, 0])}^{\frac{1}{2}}.
    \label{eqn:ricrhalf}
  \end{align}
  \label{thm:lcRm}
\end{theorem}

   \begin{proof}
   Let $v=|Ric|$. By the evolution equation of Ricci curvature along the Ricci flow and the fact $|Rm| \leq \frac{1}{m^2}<<1$,
   we have the inequality
   \begin{align}
     -\Delta v + \D{v}{t} \leq v. \label{eqn:v}
   \end{align}
   Let $\eta$ be a cutoff function which vanishes outside $D=\Omega \times [-1, 0]$
   and equals $1$ in $D'=\Omega' \times [-\frac{1}{2}, 0]$,
   where $\Omega=B_{g(0)}(x_0, 1), \; \Omega'=B_{g(0)}(x_0, \frac{1}{2})$.
   Fix $s\in[-1,0]$. Multiply inequality (\ref{eqn:v}) by $\eta^2 v^{\beta-1}$,
   then integrate the resulting inequality in $\Omega \times [-1, s]$,  we obtain
   \begin{align*}
     &\int_{-1}^s \int_{\Omega} (-\Delta v) \eta^2
     v^{\beta-1} d\mu dt + \int_{-1}^s \int_{\Omega}  \D{v}{t} \eta^2 v^{\beta-1} d\mu dt
     \leq \int_{-1}^s \int_{\Omega} \eta^2 v^{\beta} d\mu dt.
   \end{align*}
   Integration by parts yields
   \begin{align*}
    &\quad  (\beta -1) \int_{-1}^s \int_{\Omega}
      \eta^2 v^{\beta-2}|\nabla v|^2 d\mu dt
      +\int_{-1}^s \int_{\Omega} 2\eta \langle \nabla \eta, \nabla v\rangle v^{\beta-1} d\mu dt
      \\
    &\qquad + \frac{1}{\beta} \left\{ \left. \int_{\Omega} \eta^2 v^{\beta} d\mu \right|_s
     - \int_{-1}^s \int_{\Omega} 2\eta \eta' v^{\beta}d\mu dt
     + \int_{-1}^s \int_{\Omega} \eta^2 v^{\beta} R d\mu dt \right\}\\
    &\leq \int_{-1}^s \int_{\Omega} \eta^2 v^{\beta} d\mu dt.
   \end{align*}
   Note that
   $|\nabla v^{\frac{\beta}{2}}|^2 = \frac{\beta^2}{4}v^{\beta-2}|\nabla v|^2$, $|R|\leq \frac{m(m-1)}{m^2}<1$, we have
   \begin{align*}
     &\qquad 4(1-\frac{1}{\beta}) \int_{-1}^s
     \int_{\Omega} \eta^2 |\nabla v^{\frac{\beta}{2}}|^2 d\mu dt + \left. \int_{\Omega} \eta^2 v^{\beta} d\mu \right|_s \\
     &\leq \beta \int_{-1}^s \int_{\Omega} \eta^2 v^{\beta} d\mu dt
     +\int_{-1}^s \int_{\Omega} 2\eta \eta' v^{\beta}d\mu dt
     -\int_{-1}^s \int_{\Omega} \eta^2 v^{\beta} R d\mu dt \\
     &\qquad \qquad -2\beta \int_{-1}^s \int_{\Omega} \eta \langle \nabla \eta, \nabla v\rangle v^{\beta-1} d\mu dt \\
      &\leq 2\beta \int_{-1}^s \int_{\Omega} \eta^2 v^{\beta} d\mu dt
     +\int_{-1}^s \int_{\Omega} 2\eta \eta' v^{\beta}d\mu dt\\
     &\qquad \qquad + \beta \epsilon^2 \int_{-1}^s \int_{\Omega} \eta^2 v^{\beta-2}|\nabla v|^2d\mu dt
     + \beta \epsilon^{-2} \int_{-1}^s \int_{\Omega} v^{\beta} |\nabla \eta|^2 d\mu dt\\
     &=2\beta \int_{-1}^s \int_{\Omega} \eta^2 v^{\beta} d\mu dt
     +\int_{-1}^s \int_{\Omega} 2\eta \eta' v^{\beta}d\mu dt\\
     &\qquad \qquad + \frac{4}{\beta} \epsilon^2 \int_{-1}^s \int_{\Omega} \eta^2 |\nabla v^{\frac{\beta}{2}}|^2 d\mu dt
     + \beta \epsilon^{-2} \int_{-1}^s \int_{\Omega} v^{\beta} |\nabla \eta|^2 d\mu dt.
   \end{align*}
   Choose $\epsilon=\sqrt{\frac{\beta-1}{2}}$.  It follows from the previous inequality that
   \begin{align*}
     &\qquad 2(1-\frac{1}{\beta}) \int_{-1}^s
      \int_{\Omega} \eta^2 |\nabla v^{\frac{\beta}{2}}|^2 d\mu dt + \left. \int_{\Omega} \eta^2 v^{\beta} d\mu \right|_s\\
     &\leq 2\beta \int_{-1}^s \int_{\Omega} \eta^2 v^{\beta} d\mu dt
     +\int_{-1}^s \int_{\Omega} 2\eta \eta' v^{\beta}d\mu dt
     +\frac{2\beta}{\beta-1} \int_{-1}^s \int_{\Omega} v^{\beta} |\nabla \eta|^2 d\mu dt.
    \end{align*}
    Since $|\nabla (\eta v^{\frac{\beta}{2}})| \leq 2\eta^2 |\nabla v^{\frac{\beta}{2}}|^2 + 2v^{\beta} |\nabla \eta|^2$,
    we have
   \begin{align*}
     &\qquad (1-\frac{1}{\beta}) \int_{-1}^s
      \int_{\Omega} |\nabla (\eta v^{\frac{\beta}{2}})|^2 d\mu dt + \left. \int_{\Omega} \eta^2 v^{\beta} d\mu \right|_s\\
     &\leq 2\beta \int_{-1}^s \int_{\Omega} \eta^2 v^{\beta} d\mu dt
     +\int_{-1}^s \int_{\Omega} 2\eta \eta' v^{\beta}d\mu dt
     +2(\frac{\beta}{\beta-1} +\frac{\beta-1}{\beta}) \int_{-1}^s \int_{\Omega} v^{\beta} |\nabla \eta|^2 d\mu dt.
    \end{align*}
    Fix $\beta \geq 2$. We have
    \begin{align*}
     &\qquad \int_{-1}^s \int_{\Omega} |\nabla (\eta v^{\frac{\beta}{2}})|^2 d\mu dt + \left. 2\int_{\Omega} \eta^2 v^{\beta} d\mu \right|_s\\
     &\leq  4\beta \int_{-1}^s \int_{\Omega} \eta^2 v^{\beta} d\mu dt
     +4\int_{-1}^s \int_{\Omega} \eta \eta' v^{\beta}d\mu dt
     +12 \int_{-1}^s \int_{\Omega} v^{\beta} |\nabla \eta|^2 d\mu dt\\
     &\leq 6 \beta \left\{ \int_{-1}^s \int_{\Omega} \eta^2 v^{\beta} d\mu dt + \int_{-1}^s \int_{\Omega} (\eta \eta' +|\nabla \eta|^2) v^{\beta}d\mu dt
         \right\}.
    \end{align*}
    In particular, the following two inequalities hold.
    \begin{align*}
     &\int_{-1}^s \int_{\Omega} |\nabla (\eta v^{\frac{\beta}{2}})|^2 d\mu dt
        \leq 6 \beta \left\{ \int_{-1}^0 \int_{\Omega} \eta^2 v^{\beta} d\mu dt + \int_{-1}^0 \int_{\Omega} (\eta \eta' +|\nabla \eta|^2) v^{\beta}d\mu dt
         \right\},\\
     &\max_{0 \leq s \leq 1} \left. \int_{\Omega} \eta^2 v^{\beta} d\mu \right|_s
       \leq 6 \beta \left\{ \int_{-1}^0 \int_{\Omega} \eta^2 v^{\beta} d\mu dt + \int_{-1}^0 \int_{\Omega} (\eta \eta' +|\nabla \eta|^2) v^{\beta}d\mu dt
         \right\}.	
    \end{align*}
   The parabolic Sobolev inequality implies
\begin{align}
  \iint_D (\eta v^{\frac{\beta}{2}})^{\frac{2(m+2)}{m}} d\mu dt
   &\leq \sigma \left(\max_{0 \leq s \leq 1}
   \left. \int_{\Omega} \eta^2 v^{\beta} d\mu \right|_s \right)^{\frac{2}{m}}
   \left( \iint_D |\nabla (\eta v^{\frac{\beta}{2}})|^2 d\mu dt \right) \nonumber\\
   &\leq \sigma \left\{ 6\beta \iint_D (\eta^2 + \eta \eta' + |\nabla \eta|^2) v^{\beta} d\mu dt \right\}^{\frac{m+2}{m}}.
   \label{eqn:baseforit}
\end{align}

Now we consider the choice of cutoff function $\eta$.
For every $k \in \Z^+ \cup \left\{ 0 \right\}$, we define
\begin{align}
 t_k \triangleq -\frac12 - \frac{1}{2^{k+1}}, \quad
 &r_k \triangleq \frac12 + \frac{1}{2^{k+1}}, \\
 \Omega_k \triangleq B_{g(0)}(x_0,r_k), \quad
 &D_k \triangleq \Omega_k\times [t_k,1].
 \label{eqn:domains}
\end{align}
Clearly, $D' \subset \cdots D_k \subset D_{k-1} \subset \cdots D_1 \subset D_0=D$.
Let $\phi$ be a smooth cutoff function with the following properties.
\begin{align*}
    \phi(t)=
    \begin{cases}
      0, & \mbox{if } t \leq 0, \\
      1, & \mbox{if } t \geq 1.
    \end{cases}
    \quad
    0 \leq \phi' \leq 2.
\end{align*}
Let $\psi=1-\phi$.   Define
\begin{align*}
  \phi_k(t) \triangleq \phi(\frac{t-t_{k-1}}{t_k-t_{k-1}}), \quad
  \psi_k(s) \triangleq \psi(\frac{s-r_k}{r_{k-1}-r_k})
\end{align*}
for every $k \in \Z^+$. Define cutoff functions $\eta_k$ by
\begin{align*}
    \eta_k(x,t) = \phi_k(t) \psi_k(d_{g(0)}(x,x_0)).
\end{align*}
By this definition, it is clear that $\eta_k \equiv 1$ on $D_k$ and $\eta_k \equiv 0$ outside $D_{k-1}$.
Moreover, using the fact that $|Ric| \leq \frac{m-1}{m^2}$ on $X \times [-1, 0]$, we deduce
\begin{align}
  |\nabla \eta_k| \leq 2^{k+3},    \quad |\D{\eta_k}{t}| \leq 2^{k+2}.
  \label{eqn:etabd}
\end{align}

\begin{figure}[h]
 \begin{center}
  \psfrag{A}[c][c]{$y=\phi(t)$}
  \psfrag{B}[l][l]{$y=\psi(s)$}
  \psfrag{y}[c][c]{$y$}
  \psfrag{t}[c][c]{$t$}
  \psfrag{s}[c][c]{$s$}
  \psfrag{O}[l][l]{$(0,0)$}
  \psfrag{1}[c][c]{$1$}
  \includegraphics[width=0.8\columnwidth]{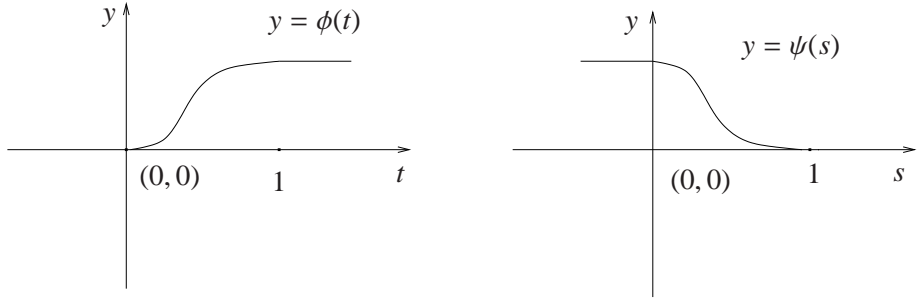}
  \caption{basic cutoff functions}
  \label{fig:basic}
 \end{center}
\end{figure}

Let $\lambda=\frac{m+2}{m}$. Inequality (\ref{eqn:baseforit}) and inequality (\ref{eqn:etabd}) imply
   \begin{align*}
     \norm{v^{\beta}}{L^{\lambda}(D_k)}& =  \norm{\eta_k^2 v^{\beta}}{L^{\lambda}(D_k)}\\
     &\leq  \norm{\eta_k^2 v^{\beta}}{L^{\lambda}(D_{k-1})} \\
     &\leq C_1 \sigma^{\frac{1}{\lambda}}  \cdot (6\beta) \cdot \iint_{D_{k-1}} (\eta^2 +\eta \eta' + |\nabla \eta|^2) v^{\beta} d\mu dt\\
     &\leq C_2 \sigma^{\frac{1}{\lambda}} \beta 4^k \iint_{D_{k-1}} v^{\beta} d\mu dt.
   \end{align*}
It follows that
   \begin{align*}
     \norm{v}{L^{\lambda \beta}(D_k)} \leq (C_2\sigma^{\frac{1}{\lambda}})^{\frac{1}{\beta}}
     \cdot \beta^{\frac{1}{\beta}} \cdot 4^{\frac{k}{\beta}} \cdot \norm{v}{L^{\beta}(D_{k-1})}.
   \end{align*}
Let $\beta=2 \lambda^{k-1}$ for $k \geq 1$, we have
   \begin{align*}
     \norm{v}{L^{2\lambda^k}(D_k)} &\leq (C_2\sigma^{\frac{1}{\lambda}})^{\frac{1}{2} \lambda^{1-k}}
     \cdot 2^{\frac{1}{2} \lambda^{1-k}} \cdot  \lambda^{\frac{k-1}{2} \lambda^{1-k}}
     \cdot 4^{\frac{k}{2} \lambda^{1-k}} \norm{v}{L^{2\lambda^{k-1}}(D_{k-1})}.
   \end{align*}
Iteration of this inequality yields
   \begin{align*}
     \norm{v}{L^{\infty}(D')}   &\leq (C_2\sigma^{\frac{1}{\lambda}})^{\sum_{k=1}^{\infty} \frac{1}{2} \lambda^{1-k}} \cdot
      2^{\sum_{k=1}^{\infty} \frac{1}{2} \lambda^{1-k}} \cdot \lambda^{\sum_{k=1}^{\infty} \frac{k-1}{2} \lambda^{1-k}}
      \cdot 4^{\sum_{k=1}^{\infty} \frac{k}{2} \lambda^{1-k}}\norm{v}{L^2(D)}\\
      &=C(m, \sigma) \norm{v}{L^2(D)}.
   \end{align*}
Recall that $v=|Ric|$, we have obtained
\begin{align}
  \sup_{D'} |Ric| \leq C(m, \sigma) \left \{\iint_D |Ric|^2 d\mu dt \right\}^{\frac{1}{2}}.
  \label{eqn:ricbdl2}
\end{align}

Choose cutoff function $\tilde{\eta}(y,t)=\psi(d_{g(t)}(y, x_0)-2)$.  It is easy to check
$\Omega=B_{g(0)}(x_0, 1) \subset B_{g(t)}(x_0, 2)$ for every $-1 \leq t \leq 0$. Therefore $\tilde{\eta} \equiv 1$
on $D=\Omega \times [-1, 0]$. From the evolution equation of scalar curvature
$\D{}{t} R = \Delta R + 2|Ric|^2$, we have
\begin{align}
  2\iint_D |Ric|^2 d\mu dt & \leq \int_{-1}^0 \int_X 2|Ric|^2 \tilde{\eta} d\mu dt \nonumber \\
  &=\int_{-1}^0 \int_X (\D{R}{t} - \Delta R) \tilde{\eta} d\mu dt \nonumber\\
  &=\int_{-1}^0 \D{}{t} \left\{ \int_X \tilde{\eta} R d\mu \right\} dt
  +\int_{-1}^0 \int_X (-\D{\tilde{\eta}}{t}-\Delta \tilde{\eta} +\tilde{\eta}R)R d\mu dt \nonumber\\
  &=\left. \left\{ \int_{B_{g(t)(x_0, 3)}} \tilde{\eta} R d\mu \right\} \right|_{-1}^0
  +\int_{-1}^0 \int_{B_{g(t)}(x_0, 3)} (-\D{\tilde{\eta}}{t}-\Delta \tilde{\eta} +\tilde{\eta}R)R d\mu dt.
  \label{eqn:ricr}
\end{align}
In order to estimate the last term, we calculate
\begin{align*}
  \left|-\D{}{t} \tilde{\eta} -\Delta \tilde{\eta} + R \tilde{\eta} \right|
  =\left|-\psi' \D{d}{t} - (\psi''|\nabla d|^2 +\psi' \Delta d) + R \tilde{\eta} \right|
  \leq
  \begin{cases}
   &C(m) \left\{ \left|\D{d}{t} \right| + \left|\Delta d \right| + 1 \right\}, \quad \textrm{if} \; 2 \leq d \leq 3.\\
   &|R| \leq \frac{m(m-1)}{m^2}, \quad \textrm{if} \; d \leq 2.
  \end{cases}
\end{align*}
Here $d=d_{g(t)}(\cdot , x_0)$. Since $inj_{g(t)}(x_0) \geq 3$ and $|Rm|_{g(t)}(x) \leq \frac{1}{m^2}$
whenever $x \in B_{g(t)}(x_0, 100)$, $t \in [-1, 0]$, Hessian comparison theorem implies
\begin{align*}
   \left|-\D{}{t} \tilde{\eta} -\Delta \tilde{\eta} + R \tilde{\eta} \right| \leq C(m).
\end{align*}
Plug this into equation (\ref{eqn:ricr}) yields
\begin{align}
  \iint_D |Ric|^2 d\mu dt &\leq
  \frac12 \left\{ \left|\int_{B_{g(0)}(x_0,3)} \tilde{\eta}R d\mu \right|
  +\left|\int_{B_{g(-1)(x_0,3)}} \tilde{\eta}R d\mu \right|
  +C(m) \int_{-1}^0 \int_{B_{g(t)}(x_0, 3)} |R| d\mu dt\right\} \nonumber\\
 &\leq   \frac12 \left\{ \int_{B_{g(0)}(x_0,10)} |R|_{g(0)} d\mu_{g(0)}
   + \int_{B_{g(0)(x_0,10)}} |R|_{g(-1)} d\mu_{g(-1)}
  +C(m) \int_{-1}^0 \int_{B_{g(0)}(x_0, 10)} |R| d\mu dt \right\} \nonumber\\
  &\leq C(m) \sup_{B_{g(0)}(x_0, 10) \times [-1,0]} |R|,
  \label{eqn:ricrg}
\end{align}
where we used the fact that $B_{g(t)}(x_0,3) \subset B_{g(0)}(x_0, 10)$ for every $-1 \leq t \leq 0$.
Combine equation (\ref{eqn:ricrg}) with equation (\ref{eqn:ricbdl2}), we obtain inequality (\ref{eqn:ricrhalf}).
\end{proof}

\begin{corollary}
 There is a constant $A_0=A_0(m, \kappa)$ such that the following property holds.

 If $\left\{ (X, g(t)), -\frac{1}{8} \leq t \leq 0 \right\} \in \mathcal{M}(m, \kappa, [-\frac{1}{8}, 0])$, then
 \begin{align}
   P(0) \leq A_0 \sqrt{\sup_{t \in [-\frac{1}{8}, 0]} O(t)}.
   \label{eqn:a12_1}
 \end{align}
  \label{cly:c3_1}
\end{corollary}

\begin{proof}
  By the uniform curvature bound and non-collapsed condition of all flows in $\mathcal{M}(m, \kappa, [-\frac{1}{8}, 0])$, we see that there exist constants
  $\sigma_0$ and $c_0$ such that
  \begin{itemize}
    \item $( B_{g(t)}(x, 1), g(t))$ has a uniform Sobolev constant $\sigma_0$ for every $(x, t)
     \in X \times [-\frac{1}{8}, 0]$.
    \item $inj_{g(t)}(x) \geq c_0$ for every $(x, t) \in X \times [-\frac{1}{8}, 0]$.
  \end{itemize}
  Let $A=\max\left\{ \frac{3}{c_0}, 100m \right\}$, $h(t)=A^{2} g(A^{-2} t)$.   Based at every point $x_0 \in X$,
  the flow  $\left\{ (X^m, h(t)), -1 \leq t \leq 0 \right\}$ satisfies all the requirements in Theorem~\ref{thm:lcRm}.  Therefore, we have
  \begin{align*}
    |Ric|_{h(0)}(x_0) &\leq \sup_{B_{h(0)}(x_0, \frac{1}{2}) \times [-\frac{1}{2}, 0]} |Ric|_{h}\\
      &\leq C \norm{R}{L^{\infty}(B_{h(0)}(x_0, 10) \times [-1, 0])}^{\frac{1}{2}}\\
      &\leq C\left( \sup_{t \in [-1, 0]} O_h(t) \right)^{\frac{1}{2}}.
  \end{align*}
  It follows that
  \begin{align*}
   &\qquad P_h(0) \leq C\left( \sup_{t \in [-1, 0]} O_h(t) \right)^{\frac{1}{2}}, \\
   &\Rightarrow  A^{-2}P_g(0) \leq C \left(A^{-2} \sup_{t \in [-A^{-2}, 0]} O_g(t)\right)^{\frac{1}{2}}
   \leq  CA^{-1} \left( \sup_{t \in [-\frac{1}{8}, 0]} O_g(t) \right)^{\frac{1}{2}}, \\
   &\Rightarrow  P_g(0) \leq CA \left( \sup_{t \in [-\frac{1}{8}, 0]} O_g(t) \right)^{\frac{1}{2}}
   =A_0 \left( \sup_{t \in [-\frac{1}{8}, 0]} O_g(t) \right)^{\frac{1}{2}},
  \end{align*}
  where we define $A_0=CA$. Clearly, $A_0=A_0(m, \kappa)$.
\end{proof}


\begin{remark}
 After scaling, inequality (\ref{eqn:a12_1}) can be regarded as an estimate of the type $|Ric| \leq \sqrt{|Rm||R|}$.
 If $|Rm|$ has a big norm compared to $|R|$, then we have an inequality of the type
 $|\overset{\circ}{Ric}| \leq \sqrt{|W||R|}$ where $W$ is the Weyl tensor.
 This estimate is then of the similar spirit of the main estimates in~\cite{Kno} and~\cite{CaX}.
 \label{rmk:c1_2}
\end{remark}

 \begin{corollary}
  Suppose $\left\{ (X, g(t)), -1 \leq t <0 \right\} \in \mathcal{L}(m, \kappa, [-1, 0))$,  $t=0$ is the singular time.
  If $\displaystyle \limsup_{t \to 0} \sqrt{O(t)Q(t)} |t|=C$, then we have
  \begin{align}
    Q(t)= o(|t|^{-\lambda}),
  \label{eqn:b23_7}
  \end{align}
  where $\lambda>\frac{1}{\log_2 (1+\frac{\epsilon_0}{\sqrt{2}A_0C})}$, $\epsilon_0$ is the constant
  in Theorem~\ref{thm:b23_1}, $A_0$ is the constant in Corollary~\ref{cly:c3_1}.
  \label{cly:b24_2}
\end{corollary}

\begin{proof}

  Fix $\lambda>\frac{1}{\log_2 (1+\frac{\epsilon_0}{\sqrt{2}A_0C})}$.  Choose $\delta>0$ such that
  \begin{align}
    \lambda>\frac{1}{\log_2 (1+\frac{\epsilon_0}{\sqrt{2}A_0(C+\delta)})}.
   \label{eqn:b24_8}
  \end{align}
  Since $\displaystyle \limsup_{t \to 0} \sqrt{O(t)Q(t)} |t|=C$ and $\displaystyle \lim_{t \to 0}Q(t)=\infty$,
  we can find a time $t_0<0$ with the following properties
  \begin{align}
    \begin{cases}
    & Q(t)O(t)t^2 < (C+\delta)^2, \quad \forall \; t \in [t_0, 0],\\
    & Q(t) \leq Q(t_0), \quad \forall \; t \in [-1, t_0].
  \end{cases}
  \label{eqn:b23_double}
  \end{align}
  Since $Q(t)$ satisfies the inequality
  \begin{align*}
    \D{}{t} Q^2 \leq 16Q^3,
  \end{align*}
  maximum principle implies that
  \begin{align*}
    Q(t) \geq \frac{1}{2}Q(t_0),  \quad \forall \; t \in [t_0-\frac{1}{8}Q^{-1}(t_0), t_0].
  \end{align*}
  Combine this with second inequality of (\ref{eqn:b23_double}), we have
  \begin{align}
   Q(t_0) \geq  Q(t) \geq \frac{1}{2}Q(t_0),  \quad \forall \; t \in [t_0-\frac{1}{8}Q^{-1}(t_0), t_0].
   \label{eqn:b23_2}
  \end{align}

  Let $t_1$ be the first time such that $\left| \log \frac{Q(t)}{Q(t_0)} \right|$ achieves $\log 2$.
  According to inequality (\ref{eqn:b23_2}) and the definition of $t_1$, we have
  \begin{align}
    Q(t) \geq \frac{1}{2} Q(t_0), \quad \forall \; t \in [t_0-\frac{1}{8}Q^{-1}(t_0), t_1].
    \label{eqn:b23_3}
  \end{align}
  On the other hand, we have
  \begin{align}
    Q(t) \leq 2Q(t_0),  \quad \forall \; t \in [-1, t_1].
    \label{eqn:b23_4}
  \end{align}
  By definition of $t_1$, $Q(t_1)=2Q(t_0)$ or $\frac{1}{2}Q(t_0)$. Clearly, we have
  \begin{align}
    Q(t_1) \geq \frac{1}{2}Q(t_0).     \label{eqn:b23_5}
  \end{align}
  Define $h(t)=Q(t_0)g(Q^{-1}(t_0)t+t_1)$. In view of (\ref{eqn:b23_4}) and (\ref{eqn:b23_5}), we have
  $\left\{ (X^m, h(t)), -1 \leq t \leq 0 \right\} \in \mathcal{M}(m, \kappa)$.
  For more general $t$,  $h(t)$ satisfies the following properties.
  \begin{align*}
   \begin{cases}
   & Q_{h}(t) \leq 2, \quad \forall \; t \in [Q(t_0)(-1-t_1), 0].\\
   & Q_{h}(t) \geq \frac{1}{2}, \quad \forall \; t \in [Q(t_0)(t_0-t_1)-\frac{1}{8}, 0].\\
   & Q_{h}(t)O_{h}(t) \left(Q(t_0)|t_1| +|t| \right)^2 <(C+\delta)^2, \quad \forall \; t \in [Q(t_0)(t_0-t_1)-\frac{1}{8}, 0].
   \end{cases}
  \end{align*}
  It follows that
  \begin{align*}
   \begin{cases}
    & \frac{1}{2} \leq Q_{h}(t) \leq 2, \quad \forall \; t \in [Q(t_0)(t_0-t_1)-\frac{1}{8}, 0].\\
    & O_{h}(t)< \frac{(C+\delta)^2}{(Q(t_0)|t_1|+|t|)^2 Q_{h}(t)} \leq \frac{2(C+\delta)^2}{(Q(t_0)|t_1|+|t|)^2}
    \leq \frac{2(C+\delta)^2}{Q^2(t_0) |t_1|^2}, \quad \forall \; t \in [Q(t_0)(t_0-t_1)-\frac{1}{8}, 0].
   \end{cases}
  \end{align*}
  By Corollary~\ref{cly:c3_1}, the estimate of $\displaystyle P_h(t)=\sup_X |Ric|_{h(t)}$ follows.
  \begin{align}
    P_{h}(t) \leq A_0 \sqrt{ \frac{2(C+\delta)^2}{Q^2(t_0)|t_1|^2}} =\frac{\sqrt{2}A_0(C+\delta)}{Q(t_0)|t_1|},
    \quad \forall  \; t \in [Q(t_0)(t_0-t_1), 0].
  \label{eqn:b24_9}
  \end{align}
  This yields that
  \begin{align}
   & \quad \int_{Q(t_0)(t_0-t_1)}^{0} P_{h}(t) dt \leq \sqrt{2}A_0(C+\delta) \cdot \frac{|t_1-t_0|}{|t_1|}= \sqrt{2}A_0(C+\delta) \left( \frac{|t_0|}{|t_1|}-1 \right), \nonumber\\
   & \Rightarrow
   \frac{|t_0|}{|t_1|} \geq 1+ \frac{1}{\sqrt{2}A_0(C+\delta)} \int_{Q(t_0)(t_0-t_1)}^{0} P_{h}(t) dt.  \label{eqn:b22_1}
  \end{align}
  By Lemma~\ref{lma:b23_1}, the fact $\left|\log \frac{Q(t_1)}{Q(t_0)} \right|=\log 2$ implies that
  \begin{align}
    \int_{Q(t_0)(t_0-t_1)}^{0} P_{h}(t) dt > \epsilon_0. \label{eqn:b22_2}
  \end{align}
  Combine inequality (\ref{eqn:b22_1}) and (\ref{eqn:b22_2}), we obtain
  \begin{align}
    \frac{|t_1|}{|t_0|} \leq \frac{1}{1+\frac{\epsilon_0}{\sqrt{2}A_0(C+\delta)}}. \label{eqn:b22_3}
  \end{align}
  Let $s_1$ be the first time such that $Q(t)$ achieves $2Q(t_0)$. Clearly, we have
  \begin{align}
     t_0 < t_1 \leq s_1 <0, \quad
     \Rightarrow
     \quad
     \frac{|s_1|}{|t_0|} \leq \frac{|t_1|}{|t_0|} \leq \frac{1}{1+\frac{\epsilon_0}{\sqrt{2}A_0(C+\delta)}}.
  \label{eqn:b24_10}
  \end{align}
  For each integer $i \geq 0$, define $s_i=\inf \{t| Q(t)=2^iQ(t_0)\}$.  Clearly, $s_0=t_0$.
  Inequality (\ref{eqn:b24_10}) can be written as
  \begin{align*}
   \frac{|s_1|}{|s_0|} \leq \frac{1}{1+\frac{\epsilon_0}{\sqrt{2}A_0(C+\delta)}}.
  \end{align*}
  Generally, we have
  \begin{align}
    \frac{|s_{i+1}|}{|s_i|} \leq \frac{1}{1+\frac{\epsilon_0}{\sqrt{2}A_0(C+\delta)}}, \label{eqn:b22_4}
  \end{align}
  which yields
  \begin{align*}
    \frac{|s_i|}{|t_0|} \leq \left( \frac{1}{1+\frac{\epsilon_0}{\sqrt{2}A_0(C+\delta)}} \right)^i.
  \end{align*}
  Therefore, we have
  \begin{align}
    \lim_{i \to \infty} Q(s_i)|s_i|^{\lambda}
    &\leq \lim_{i \to \infty}  Q(t_0)|t_0|^{\lambda} \cdot \left( \frac{2}{(1+\frac{\epsilon_0}{\sqrt{2}A_0(C+\delta)})^{\lambda}} \right)^i.
  \label{eqn:b23_6}
  \end{align}
  In view of inequality (\ref{eqn:b24_8}), we have
  \begin{align*}
    \frac{2}{(1+\frac{\epsilon_0}{\sqrt{2}A_0(C+\delta)})^{\lambda}}<1.
  \end{align*}
  This yields that $\displaystyle \lim_{i \to \infty}Q(s_i)|s_i|^{\lambda}=0$.
  Since for every $t \in [s_i, s_{i+1}]$, one has
  \begin{align*}
    Q(t)|t|^{\lambda} \leq Q(s_{i+1})|s_i|^{\lambda} \leq 2 Q(s_i)|s_i|^{\lambda} \to 0, \quad \textrm{as} \; i \to \infty.
  \end{align*}
  It follows that
  \begin{align*}
    \limsup_{t \to 0} Q(t)|t|^{\lambda}=0, \quad
    \Leftrightarrow \;  Q(t)=o(|t|^{-\lambda}).
  \end{align*}

\end{proof}

\begin{corollary}
 Suppose $\left\{ (X, g(t)), -1 \leq t <0 \right\} \in \mathcal{L}(m, \kappa, [-1, 0))$,  $t=0$ is the singular time.
 Then
  \begin{align}
    \limsup_{t \to 0} \sqrt{O(t)Q(t)} |t| \geq \frac{\epsilon_0}{\sqrt{2}A_0},
  \label{eqn:b23_8}
  \end{align}
 where $\epsilon_0=\epsilon_0(m, \kappa)$ is the constant in Theorem~\ref{thm:b23_1},
 $A_0$ is the constant in Corollary~\ref{cly:c3_1}.
  \label{cly:b22_1}
\end{corollary}

\begin{proof}
  Suppose $\displaystyle \limsup_{t \to 0} \sqrt{O(t)Q(t)} |t|=C$.
  By Corollary~\ref{cly:b24_2}, we have
  \begin{align}
    \lim_{t \to 0} Q(t)|t|^{\lambda}=0,
    \label{eqn:b22_5}
  \end{align}
  whenever $\lambda >\frac{1}{\log_2 (1+\frac{\epsilon_0}{\sqrt{2}A_0C})}$.

  On the other hand, since $t=0$ is a singular time, we have
  \begin{align}
    Q(t) \geq \frac{1}{8|t|}.
    \label{eqn:b22_6}
  \end{align}
  Therefore $\displaystyle \lim_{t \to 0} |t|^{\lambda-1}=0$ whenever $\lambda >\frac{1}{\log_2 (1+\frac{\epsilon_0}{\sqrt{2}A_0C})}$.
  It forces that
  \begin{align*}
    \frac{1}{\log_2 (1+\frac{\epsilon_0}{\sqrt{2}A_0C})} \geq 1, \quad \Rightarrow
    \frac{\epsilon_0}{\sqrt{2}A_0C} \leq 1, \quad \Rightarrow \quad
    C \geq \frac{\epsilon_0}{\sqrt{2}A_0}.
  \end{align*}
\end{proof}

\section{Applications of the curvature estimates}

\begin{proof}[Proof of Theorem~\ref{thmin:et1}]
   According to Perelman's no-local-collapsing theorem (c.f.~\cite{Pe1},~\cite{KL}),
   we obtain that the flow $\left\{ (X, g(t)), 0 \leq t<T \right\}$
  is $\kappa$-non-collapsed for some constant $\kappa$ depending only on $g(0)$ and $T$.
  Therefore, $\left\{ (X, g(t)), 0 \leq t<T \right\} \in \mathcal{L}(m, \kappa, [0, T))$.
  Corollary~\ref{cly:a6_2} applies and the Theorem is proved by letting
  $\eta_1= \frac{\epsilon_0}{\log 2}$.
\end{proof}

\begin{proof}[Proof of Theorem~\ref{thmin:et2}]
  It follows from Perelman's no-local-collapsing theorem and Corollary~\ref{cly:b24_1}.
\end{proof}

\begin{proof}[Proof of Theorem~\ref{thmin:et3}]
  It follows from Perelman's no-local-collapsing theorem and Corollary~\ref{cly:b22_1}.
\end{proof}

\begin{proof}[Proof of Theorem~\ref{thmin:et4}]
  It follows from Perelman's no-local-collapsing theorem and Corollary~\ref{cly:b24_2}.
\end{proof}

\begin{proof}[Proof of Theorem~\ref{thmin:ssolitongap}]
  It follows from Corollary~\ref{cly:a6_2}, Corollary~\ref{cly:b22_1}, and the fact that
  every shrinking soliton can be expanded to an ancient solution.
\end{proof}

\begin{corollary}
   Suppose $\left\{ (X, g(t)), 0 \leq t<T \right\}$ is a Ricci flow solution on a closed manifold, $t=T$ is a singular time.
 Then
 \begin{align}
   \int_0^T \sup_X |Ric|_{g(t)} dt =\infty.
   \label{eqn:et_ricint}
 \end{align}
  \label{cly:et_ricint}
\end{corollary}

\begin{proof}
  It is an application of Perelman's no-local-collapsing theorem and Corollary~\ref{cly:a7_1}.
\end{proof}

Starting from either Theorem~\ref{thmin:et1} or Corollary~\ref{cly:et_ricint}, we obtain the famous extension theorem of N.Sesum.
\begin{corollary}[N.Sesum, c.f.~\cite{Se}]
   Suppose $\left\{ (X, g(t)), 0 \leq t<T \right\}$ is a Ricci flow solution on a closed manifold, $t=T$ is a singular time.
 Then
 \begin{align*}
   \sup_{X \times [0, T)} |Ric|=\infty.
 \end{align*}
  \label{cly:et_sesum}
\end{corollary}

\begin{corollary}
  Suppose $\left\{ (X, g(t)), 0 \leq t<T \right\}$ is a Ricci flow solution, $X$ is a closed manifold,
  $t=T$ is a singular time of type-I.  Then
  \begin{align}
    \limsup_{t \to T} |T-t| \sup_X |R|_{g(t)} >0.
    \label{eqn:et_typeI}
  \end{align}
  In other words, the norm of scalar curvature blows up at least at the rate of type-I.
  \label{cly:et_typeI}
\end{corollary}

\begin{proof}
  By Theorem~\ref{thmin:et3}, there exists a sequence of times $t_i \to T$ such that
  \begin{align*}
    \lim_{i \to \infty} |T-t_i| \sqrt{\sup_X |Rm|_{g(t_i)}} \cdot \sqrt{\sup_X |R|_{g(t_i)}} \geq \eta_3.
  \end{align*}
  Since the singularity is of type-I, we have
  \begin{align*}
    |T-t_i|\sup_X |Rm|_{g(t_i)}<C<\infty
  \end{align*}
  for every $i$. It follows that
  \begin{align*}
   &\quad \lim_{i \to \infty} \sqrt{T-t_i} \cdot \sqrt{\sup_X |R|_{g(t_i)}} \geq \frac{\eta_3}{\sqrt{C}}>0, \\
   &\Rightarrow \lim_{i \to \infty} |T-t_i| \sup_X |R|_{g(t_i)} \geq \frac{\eta_3^2}{C}>0, \\
   &\Rightarrow \limsup_{t \to T} |T-t|\sup_X |R|_{g(t)}>0.
  \end{align*}
\end{proof}

Since the type-I condition is strong, Corollary~\ref{cly:et_typeI} can be improved.
\begin{proposition}
   Suppose $\left\{ (X, g(t)), 0 \leq t<T \right\}$ is a Ricci flow solution, $X$ is a closed manifold,
  $t=T$ is a singular time of type-I.  Then
  \begin{align}
    \liminf_{t \to T} |T-t| \sup_X |R|_{g(t)} >0.
    \label{eqn:et_typeI_2}
  \end{align}
  \label{prn:et_typeI_2}
\end{proposition}

\begin{proof}
 Otherwise, we have a sequence of times $t_i \to T$ such that
 \begin{align*}
   \lim_{i \to \infty} |T-t_i|\sup_X |R|_{g(t_i)}=0.
 \end{align*}
 Let $g_i=|T-t_i|^{-1}g(|T-t_i|t+t_i)$, $\displaystyle C=\sup_{0 \leq t< T} |T-t|\sup_X|Rm|_{g(t)}<\infty$.
 Clearly, the flow $\left\{ (X, g_i(t)), \frac{-t_i}{T-t_i} \leq t \leq 0 \right\}$ satisfies
 \begin{align*}
   \sup_X |Rm|_{g_i(t)} \leq \frac{C}{1-t} \leq C, \quad \forall \; t \in [-\frac{t_i}{T-t_i}, 0].
 \end{align*}
 Moreover, each $g_i$ is $\kappa$-non-collapsed.  Let $x_i$ be the point where $\displaystyle \sup_X |Rm|_{g(t_i)}$
 is achieved.  Then we have smooth convergence:
 \begin{align*}
   \left\{ (X, x_i, g_i(t)), \frac{-t_i}{T-t_i} \leq t \leq 0 \right\} \stackrel{C^{\infty}}{\longrightarrow}
   \left\{ (X_{\infty}, x_{\infty}, g_{\infty}(t)), -\infty < t \leq 0 \right\}.
 \end{align*}
 The limit solution $g_{\infty}$ is a complete ancient Ricci flow solution with bounded curvature. It follows from maximum principle that $R \geq 0$ (c.f.Lemma 2.18 of~\cite{CLN}).  Moreover, we have
 \begin{align*}
   |R|_{g_{\infty}(0)}(x_{\infty})=\lim_{i \to \infty} |R|_{g_i(0)}(x_i) \leq \lim_{i \to \infty} \sup_X |R|_{g_i(0)}
    =\lim_{i \to \infty} |T-t_i|\sup_X |R|_{g(t_i)}=0.
 \end{align*}
 Therefore the strong maximum principle applies and we obtain $R \equiv 0$ on the whole
 space-time. By the evolution equation of the scalar curvature, we see that $Ric \equiv 0$.
 So the flow $g_{\infty}$ is ``static". For every large number $A>0$, we have $|Rm|_{g_{\infty}(-A)}(x_{\infty})=|Rm|_{g_{\infty}(0)}(x_{\infty}) \geq \frac{1}{8}>\frac{1}{9}$.
 It follows that
 \begin{align*}
   |Rm|_{g(t_i-A|T-t_i|)}(x_i) > \frac{1}{9|T-t_i|}, \Rightarrow
   (A+1)|T-t_i| \cdot |Rm|_{g(T-(A+1)|T-t_i|)}(x_i)>\frac{A+1}{9}.
 \end{align*}
 This forces that
 \begin{align*}
   C=\sup_{t \in [0, T)} |T-t| \sup_X |Rm|_{g(t)}>\frac{A+1}{9},
 \end{align*}
 for every $A>0$.  In other words, $C=\infty$, which contradicts to the assumption that the singularity is type-I.
\end{proof}

 Proposition~\ref{prn:et_typeI_2} has strong relationship to the results in~\cite{EMT}. \\

 Theorem~\ref{thmin:et3} can be used to give an alternative proof of the extension theorems in~\cite{LS3}.

\begin{corollary}[N.Le, N.Sesum, c.f.~\cite{LS3}]
  Suppose $\left\{ (X, g(t)), 0 \leq t<T \right\}$ is a Ricci flow solution, $X$ is a closed manifold of dimension $m$,
  $t=T$ is a singular time of type-I.  Then
  \begin{align}
    \int_0^T \int_X |R|^{\alpha} dvdt=\infty,  \quad \forall \; \alpha \geq \frac{m+2}{2}.
    \label{eqn:et_typeI_int}
  \end{align}
  \label{cly:et_typeI_int}
\end{corollary}

\begin{proof}
  If this statement was wrong, then we have
  \begin{align}
      \int_0^T \int_X |R|^{\alpha} dvdt<\infty,  \quad \forall \; \alpha \geq \frac{m+2}{2}.
    \label{eqn:et_typeI_int_1}
  \end{align}
  By Corollary~\ref{cly:et_typeI}, there exist a sequence of times $t_i \to T$
  and a sequence of points $x_i \in X$ such that
  \begin{align*}
    \lim_{i \to \infty} |T-t_i| |R|_{g(t_i)}(x_i)=c_0>0.
  \end{align*}
  Let $g_i(t)=|T-t_i|^{-1}g(|T-t_i|t+t_i)$, $\displaystyle C=\sup_{0 \leq t< T} |T-t|\sup_X|Rm|_{g(t)}<\infty$.
  The flow $\left\{ (X, g_i(t)), -\frac{t_i}{T-t_i} \leq t \leq 0 \right\}$ satisfies the properties.
  \begin{align}
    & \sup_X |Rm|_{g_i(t)} \leq \frac{C}{1-t}.  \label{eqn:rmbd}\\
    & |R|_{g_i(0)}(x_i) \geq c_0. \label{eqn:scalargap}\\
    & \int_{-\frac{1}{2}}^0 \int_X |R|_{g_i(t)}^{\alpha} dvdt \to 0.  \label{eqn:scalarintsmall}
  \end{align}
  Therefore, we have
  \begin{align*}
    \left\{ (X, x_i, g_i(t)), -1<t \leq 0 \right\} \stackrel{C^{\infty}}{\longrightarrow}
    \left\{ (X_{\infty}, x_{\infty}, g_{\infty}(t)), -1 <t\leq 0 \right\}.
  \end{align*}
  It follows from equation (\ref{eqn:scalargap}) that $|R|_{g_{\infty}(0)}(x_{\infty})\geq c_0>0$.  On the other hand,
  equation (\ref{eqn:scalarintsmall}) implies that $|R|_{g_{\infty}(0)}(x_{\infty})=0$. Contradiction!
\end{proof}

In light of Theorem~\ref{thmin:et1}, the main theorem in~\cite{BWa} can be improved.
\begin{theorem}
  Suppose $\left\{ (X, g(t)), 0 \leq t<T \right\}$ is a Ricci flow solution, $X$ is a closed manifold of dimension $m$,
  $t=T$ is a singular time.  Then either $\int_0^T \int_X |R|^{\frac{m+2}{2}} dv dt=\infty$, or
  \begin{align}
    \limsup_{t \to T} |T-t| \left(\sup_X |Ric_{-}|_{g(t)} \right) \geq \eta_4,
    \label{eqn:et4}
  \end{align}
  where $\eta_4=\eta_4(m, \kappa)<< \eta_1(m, \kappa)$, $\kappa$ is the non-collapsing constant of this flow,
  $Ric_{-}$ is the negative part of the Ricci tensor, $\eta_1$ is the constant in Theorem~\ref{thmin:et1}.

  In other words, if $\int_0^T \int_X |R|^{\frac{m+2}{2}}dvdt<\infty$ and $T$ is a singular time,
  then $|Ric_{-}|$ blows up at least at the rate of type-I.
  \label{thm:et_rint}
\end{theorem}

\begin{proof}

  Suppose the flow does not satisfy the second condition, i.e.,
  \begin{align}
     \limsup_{t \to T} |T-t| \left(\sup_X |Ric_{-}|_{g(t)} \right) < \eta_4.
     \label{eqn:lesseta4}
  \end{align}

  Let $A_i$ be an increasing sequence such that
  \begin{align}
    \lim_{i \to \infty} A_i = \limsup_{t \to T} |T-t|\sup_X |Ric|_{g(t)} \geq \eta_1.
    \label{eqn:Aichoice}
  \end{align}
  For each $A_i$, let $t_i$ be the first time such that $\displaystyle |T-t_i|\sup_X |Ric|_{g(t_i)}=A_i$.  Suppose
  that $\displaystyle P_i=\sup_X |Ric|_{g(t_i)}$ is achieved at the point $x_i$.  Define
  $g_i(t)=P_ig(P_i^{-1}t+t_i)$. This flow satisfies
  \begin{align*}
    \sup_X |Ric|_{g_i(t)} \leq \frac{A_i}{A_i-t} \leq 1, \quad \; \forall \; t \in [-P_it_i, 0].
  \end{align*}
  By the fact $\eta_4 << \eta_1$ and equation (\ref{eqn:Aichoice}),
  we have $\eta_4 << A_i$.   It follows from inequality (\ref{eqn:lesseta4})
  that $Ric_{g_i(t)}$ is almost nonnegative when $t \leq 0$.  In particular, at point $(x_i, 0)$, we have
  \begin{align*}
    R^2 \geq \frac{1}{2} |Ric|^2 \geq \frac{1}{2}.
  \end{align*}
  Then apply parabolic Moser iteration to the evolution equation of scalar curvature (c.f.~\cite{BWa} for more details), we have
  \begin{align*}
    \int_{-1}^0 \int_X |R|^{\frac{m+2}{2}} dvdt \geq \delta
  \end{align*}
  for some fixed constant $\delta$.  In view of the scale invariance of the integration, by taking subsequence if necessary, we have
  \begin{align*}
    \int_0^T \int_X |R|^{\frac{m+2}{2}} dv dt
    \geq \sum_{k=1}^{\infty}   \int_{t_{i_k}}^{t_{i_{k+1}}} \int_X |R|^{\frac{m+2}{2}} dv dt \geq \sum_{k=1}^{\infty} \delta =\infty.
  \end{align*}

\end{proof}

\begin{theorem}
  There exists a positive constant $\eta_5=\eta_5(m,V)$ such that the following property holds.

  Suppose $(X^m, g)$ is a complete, non-flat, gradient shrinking soliton:
  \begin{align}
    R_{ij}+f_{ij}-\frac{g_{ij}}{2}=0.   \label{eqn:gssoliton}
  \end{align}
  If $(4\pi)^{-\frac{m}{2}} \int_X e^{-f}dv \geq V$, then
  \begin{align}
    \min \left\{\sqrt{\sup_X |Rm|} \cdot \sqrt{\sup_X |R|}, \quad \sup_X |Ric| \right\} \geq \eta_5.
    \label{eqn:gssolitongap}
  \end{align}
  \label{cly:gssolitongap}
\end{theorem}

\begin{proof}
  In the gradient shrinking soliton case,
  the Gap Theorem of~\cite{MW} implies
  $\displaystyle \sup_X |Ric| \geq \frac{1}{100m^2}$. It suffices to show
   \begin{align}
    \sqrt{\sup_X |Rm|} \cdot \sqrt{\sup_X |R|} \geq \eta_5.
    \label{eqn:RmRgap}
   \end{align}
  Since $(X, g)$ is non-flat, we see that $\displaystyle \sup_X |R|>0$ by Theorem 3 of~\cite{PRS}.
  If $\displaystyle \sup_X |Rm|=\infty$, then inequality (\ref{eqn:RmRgap}) holds trivially.  So we
  assume $\displaystyle \sup_X |Rm|<\infty$.  However, according to Theorem 4.2 of~\cite{Nab}, there exists a $\kappa=\kappa(m, V)$ such that $(X, g)$ is $\kappa$-non-collapsed. Therefore, inequality (\ref{eqn:RmRgap}) follows from Theorem~\ref{thmin:ssolitongap}.
 \end{proof}

 Bing  Wang,   Department of Mathematics, Princeton University,
  Princeton, NJ 08544, USA; bingw@math.princeton.edu\\

\end{document}